\documentclass{gtpart}     
\gtart  
\usepackage{amscd,latexsym,amsfonts,epsf,graphicx}
\title{All finite groups are involved in the Mapping Class Group}

\author{Gregor Masbaum}
\givenname{Gregor}
\surname{Masbaum}
\address{Institut de Math{\'e}matiques de Jussieu (UMR 7586 du CNRS),
Case 247, 4 pl. Jussieu,
75252 Paris Cedex 5, France}
\email{masbaum@math.jussieu.fr}
\urladdr{}

\author{Alan W. Reid}
\givenname{Alan W.}
\surname{Reid}
\address{Department of Mathematics,
University of Texas, 
Austin, TX 78712, USA.}
\email{areid@math.utexas.edu}


\volumenumber{}
\issuenumber{}
\publicationyear{}
\papernumber{}
\startpage{}
\endpage{}
\doi{}
\MR{}
\Zbl{}
\received{}
\revised{}
\accepted{}
\published{}
\publishedonline{}
\proposed{}
\seconded{}
\corresponding{}
\editor{}
\version{}

\newtheorem{theorem}{Theorem}[section]    
\newtheorem{lemma}[theorem]{Lemma}         
\newtheorem{corollary}[theorem]{Corollary} 
\newtheorem{remark}[theorem]{Remark}       
       
\newtheorem{definition}[theorem]{Definition}

\newtheorem{proposition}[theorem]{Proposition}

\newtheorem{notation}[theorem]{Notation} 
\def\tr{\mbox{\rm{tr}}}

\def\PSL{\mbox{\rm{PSL}}}
\def\P{\mbox{\rm{P}}}

\def\SL{\mbox{\rm{SL}}}
\def\SO{\mbox{\rm{SO}}}
\def\SU{\mbox{\rm{SU}}}
\def\Sp{\mbox{\rm{Sp}}}
\def\PSU{\mbox{\rm{PSU}}}

\def\rk{\mbox{\rm{rk}}}
\def\GL{\mbox{\rm{GL}}}
\def\Gal{\mbox{\rm{Gal}}}

\def\Ad{\mbox{\rm{Ad}}}
\def\det{\mbox{\rm{det}}}

\def\Out{\mbox{\rm{Out}}}

\def\dim{\mbox{\rm{dim}}}

\def\qed{ $\sqcup\!\!\!\!\sqcap$}

\newcommand{\BZ}{{\bf{Z}}}
\newcommand{\BQ}{{\bf{Q}}}

\newcommand{\BC}{{\bf{C}}}
\newcommand{\Si}{{\Sigma}}

\newcommand{\BF}{{\bf{F}}}
\newcommand{\Gad}{{\mathcal G}_{ad}}

\def\Trd{\mbox{\rm{Trd}}}
\def\Nrd{\mbox{\rm{Nrd}}}

\def\mapright#1{\smash{\mathop{\longrightarrow}\limits^{#1}}}

\begin{document} 

\begin{abstract}    
Let $\Gamma_g$
denote the orientation-preserving Mapping Class Group of the genus $g\geq 1$
closed orientable surface. In this paper we show that
for fixed $g$, every finite group occurs as a quotient of a finite
index subgroup of $\Gamma_g$.
\end{abstract}

\maketitle


\section{Introduction} \label{sec1}
Throughout this paper, $\Gamma_g$ will
denote the orientation-preserving Mapping Class Group of the genus $g$
closed orientable surface. 

A group $H$ is 
{\em involved}
in a group $G$ if there exists a finite
index subgroup $K<G$ and an epimorphism from $K$ onto $H$. 
The question as to whether every finite group
is involved in $\Gamma_g$ was raised by U. Hamenst{\"a}dt in her talk at
the 2009 Georgia Topology Conference. 
The main
result of this note is the following.

\begin{theorem}
\label{main}
For all $g\geq 1$, every finite group is involved in $\Gamma_g$.\end{theorem}

Some comments are in order.  When $g=1$, $\Gamma_1 \cong \SL(2,{\bf
Z})$ and in this case the result follows since $\SL(2,{\bf Z})$
contains free subgroups of finite index (of arbitrarily large rank).
For the case of $g=2$, it is known that
$\Gamma_2$ is large  \cite{Ko}; that is to say, $\Gamma_2$ contains a finite
index subgroup that surjects a free non-abelian group, and again the
result follows.  Thus, it suffices to deal with the case when $g\geq
3$. 

Although $\Gamma_g$ is well-known to be residually finite \cite{Gro},
and therefore has a rich supply of finite quotients, 
apart from those finite quotients obtained from 

$$\Gamma_g \rightarrow \Sp(2g,{\bf Z}) \rightarrow \Sp(2g,{\bf Z}/N{\bf Z})$$

\noindent very little seems known explicitly 
about what  finite groups can arise as
quotients of $\Gamma_g$ (or of subgroups of finite index). 
Some constructions of finite quotients of finite index
subgroups of $\Gamma_g$ do appear in the literature; for example, in
\cite{DT}, \cite{FP} and \cite{Lo}. In particular, the constructions
in \cite{Lo} using Prym representations associated to finite abelian
overs of surfaces can be used to construct finite quotients that 
are
similar in spirit to what is done here.
Further information about the structure of finite index subgroups of
$\Gamma_g$ is contained in \cite{BGP} where the minimal index of a 
proper subgroup of $\Gamma_g$ is computed.

Theorem \ref{main} will follow (see \S \ref{sec4}) from our next result which gives 
many new finite simple groups of Lie type as quotients of $\Gamma_g$.
Throughout the paper, ${\bf F}_q$ will denote a finite field of order $q$, and
$\SL(N,q)$ (resp. $\PSL(N,q))$ 
will denote the finite group $\SL(N, {\bf F}_q)$ (resp. $\PSL(N, {\bf F}_q)$).

\begin{theorem}
\label{main2}
For each $g\geq 3$, there
exist
infinitely many $N$ such that
for each such $N$, there 
exist
infinitely
many primes $q$ such that $\Gamma_g$ surjects $\PSL(N,q)$.\end{theorem}

In addition we show that Theorem \ref{main2} 
also holds for the Torelli group (with $g\geq 2$).

It is worth emphasizing 
that
one cannot expect to prove Theorem \ref{main} simply using
the subgroup structure of the groups $\Sp(2g,{\bf Z}/N{\bf Z})$. The reason
for this is 
that since $\Sp(2g,{\bf Z})$ has the Congruence Subgroup 
Property (\cite{BMS}), it is well-known that not all finite groups
are involved in
$\Sp(2g,{\bf Z})$ (see \cite{LR} Chapter 4.0 for example).

An interesting feature of the proof of Theorem \ref{main} is that it
exploits the unitary representations arising 
in Topological Quantum Field Theory (TQFT) first constructed by Reshetikhin and
Turaev \cite{RT}. We actually use the so-called $\SO(3)$-TQFT
following the skein-theoretical approach of \cite{BHMV}
(see \S \ref{sec3} for a brief resum{\'e} of this).

We briefly indicate the strategy of the proof of Theorem \ref{main2}.
The unitary representations 
that we consider
are indexed
by primes $p$ congruent to $3$ modulo $4$. For each such $p$
we exhibit a group
$\Delta_g$ which is the image of a certain 
central
extension
$\widetilde{\Gamma}_g$ of $\Gamma_g$ and satisfies
$$\Delta_g\  \subset \SL(N_p,\BZ[\zeta_p])~,$$
where $\zeta_p$ is a primitive $p$-th root of unity, and $\BZ[\zeta_p]$ 
is the ring of integers in ${\bf Q}(\zeta_p)$. Moreover, the dimension
$N_p\rightarrow \infty$ as we vary $p$.
The key part of the proof is
the following. We exhibit infinitely many rational primes $q$, 
and prime ideals $\tilde q \subset \BZ[\zeta_p]$ 
satisfying $\BZ[\zeta_p]/\tilde q\simeq \BF_q$, for which
the reduction homomorphism $\pi_{\tilde q}$ 
       from $\SL(N_p,\BZ[\zeta_p])$         to $\SL(N_p,q)$
(induced by the isomorphism
$\BZ[\zeta_p]/\tilde q\simeq \BF_q$) 
restricts to a surjection
$\Delta_g\twoheadrightarrow \SL(N_p,q)$.

{}From this, it is then easy to get surjections
$\Gamma_g \twoheadrightarrow \PSL(N_p,q)$, 
which will complete
the proof. The details of how all of this is 
achieved are given in \S\ref{sec4}.

The paper is organized as follows. In \S \ref{sec2} we collect some background
on algebraic and arithmetic subgroups of (special) unitary groups, as well
as what is needed for us from Strong Approximation.  This is all
well-known, but 
we include this to help make the paper more self-contained. In \S \ref{sec3} we
discuss the (projective) unitary representations of
$\Gamma_g$ arising from $\SO(3)$-TQFT
and a density result for these representations due to Larsen and Wang
\cite{LW}. 
In \S \ref{sec4} we put the pieces together
to prove Theorems \ref{main} and \ref{main2} 
following the strategy outlined above. 
Finally, in \S \ref{sec5} we make some additional comments
about Theorem \ref{main}. In particular, 
how Theorem \ref{main} is perhaps reflective of 
some more ``rank 1'' phenomena for $\Gamma_g$.\\[\baselineskip]
\noindent{\bf Acknowledgements:}~{\em The authors wish to thank the organizers of two conferences in June 2009 at which they
first began thinking about this problem: 
"From Braid groups to Teichmuller spaces", and  "On Interactions between Hyperbolic Geometry, Quantum
Topology and Number Theory" at  C.I.R.M. Luminy and Columbia
University respectively. We also wish to thank 
Ian Agol, Mathieu Florence, and Matt Stover for helpful
conversations. We would particularly like to thank Gopal Prasad who
helped enormously in clarifying various points about algebraic groups,
their $k$-forms and fields of definition that are used in \S
\ref{sec4}.  The second author thanks Max Planck Institute for Mathematics for
its hospitality whilst working on this.

The second author was partially supported by the NSF.}
\begin{remark} {\em Whilst in the process of completing the writing
of this paper we have
learned
that similar results have recently been proved by L. Funar 
\cite{Fu}.}\end{remark}

\section{Algebraic and arithmetic aspects of unitary groups} \label{sec2}

It will be convenient to recall some of the basic background of unitary groups,
algebraic groups arising from Hermitian forms (over number fields, local
fields and finite fields), 
their arithmetic subgroups, and some aspects of
the Zariski topology that we will make use of.

We begin by fixing some notation. Throughout this paper
we will fix $p$ to be 
an odd prime, which will be assumed congruent 
to 3 modulo 4 from \S \ref{sec3} on.
Let 
$\zeta=\zeta_p$ 
denote a primitive $p$-th root of unity, $K_p$ (or
simply $K$ if no confusion will arise)
will denote the cyclotomic field ${\bf Q}(\zeta)$ and ${\mathcal O}_K$ its ring
of integers. We will let the maximal real subfield of $K_p$ be denoted
by $k=k_p$, with
corresponding ring of integers ${\mathcal O}_k$. We will assume that these fields
always come with a specific embedding into $\bf C$.
$K_p$ is a totally imaginary quadratic extension
of the totally real field $k_p$, and both are Galois extensions
of $\bf Q$.

If $G < \GL(m,{\bf C})$ is an algebraic group, and $R\subset {\bf C}$ is a subring then we will denote the $R$-points of $G$ by $G(R) = G \cap \GL(m,R)$.
We will identify $G$ with its complex points.

\subsection{} For more details about the material covered in this section
see \cite{PR}, \cite{Sh1} and \cite{Sh2}.

 First, consider the extension of fields $K/k$. Fixing an embedding of 
$K\subset {\bf C}$,
complex conjugation induces a Galois automorphism of $K$ fixing $k$
(since $\overline{\zeta}   = \zeta^{-1}$).

Now $K/k$ has a 
$k$-basis $\{1,\zeta\}$, and for $\alpha \in K$, we can
express the $k$-linear map 
       $L_z(\alpha) = z  \alpha$ 
in terms 
of the above basis. If $z=a+b\zeta$
with $a, b\in k$ then 
$L_z$ is represented by
the following element of $M(2,k)$:
 $$L_z=\begin{pmatrix}
         a & -b \\
         b & a+bt
\end{pmatrix}
$$
where $t=\zeta+\zeta^{-1}$.
 Extending the $k$-linear map $L$ in the obvious way,
it follows that $\SL(N,K)$ 
may be embedded in $\GL(2N,{\bf C})$ 
as an algebraic group
defined over $k$. Clearly, $\SL(N,K)$ maps into $\SL(2N,k)$.
Furthermore, since $\{1,\zeta\}$ generates ${\mathcal O}_K$ over ${\mathcal O}_k$,
then $\SL(N,{\mathcal O}_K)$ maps into $\SL(2N,{\mathcal O}_k)$.

Let $V=K^N$ and $H$ a non-degenerate Hermitian form on $V$. 
The {\em special unitary group}

$$\SU(V,H) = \{A\in 
\SL(N,{\bf C}):
 \overline{A}^tHA = H\}$$

\noindent 
also has the structure of an algebraic group defined over $k$ (where
$\overline{A}$ denotes 
complex conjugation of matrices.)
This is because $L_{\bar z}$ is represented by the matrix
$$L_{\bar
 z}=
\begin{pmatrix} a+bt&b\\
                  -b&a \end{pmatrix}
$$ 
so that when we embed $K$ into $M_2(k)$ using the map $L$, complex
conjugation becomes the restriction of a self-map of $M_2(k)$
defined over $k$.

We will denote this algebraic group by $\mathcal G$, and we will frequently blur the distinction between
$\SU(V,H)$ and $\mathcal G$. 

The group $\SU(V,H;{\mathcal O}_K) = \SU(V,H) \cap \SL(N,{\mathcal O}_K)$ embeds in 
$\SL(2N,k)$ as a subgroup 
commensurable with 
${\mathcal G}({\mathcal O}_k)$. 
Indeed, in this case, using the remark above regarding
the image of $\SL(N,{\mathcal O}_K)$, we deduce that
the image of $\SU(V,H;{\mathcal O}_K)$ 
is actually equal to
${\mathcal G}({\mathcal O}_k)$.
Denoting this image group by $\Gamma$, 
then $\Gamma$ is an arithmetic
subgroup of a product ${\bf SU} = \SU(p_1,q_1)\times \ldots \SU(p_s,q_s)$,
of special unitary groups that arise from $\SU(V,H)$
in the following way (see \cite{BH} and \cite{Sh1} for more details). 
Let $\sigma_1, \ldots \sigma_d$ denote the Galois embeddings of 
$k\hookrightarrow {\bf R}$ (with $\sigma_1$ chosen to be
the identity embedding).  We will
that assume that at $\sigma_1$
$$\SU(V,H;{\bf R}) = {\mathcal G}({\bf R}) \cong \SU(p_1,q_1),$$
\noindent where $p_1+q_1=N$ and $p_1,q_1>0$.  

Applying a Galois embedding $\sigma_i$ to $\mathcal G$ produces an algebraic
group defined over $\sigma_i(k)=k$ whose real points thereby determine
another special unitary group of some signature.  Assume that for $i=1,\ldots s$
this special unitary group, denoted by $\SU(p_i,q_i)$,
is not isomorphic to $\SU(N)$ ({\em i.e.,} is non-compact) and
for $i=s+1,\ldots r$ the special unitary group is isomorphic to $\SU(N)$. The
theory of arithmetic groups then shows that $\Gamma$ is an
arithmetic subgroup of ${\bf SU} = \SU(p_1,q_1)\times \ldots \SU(p_s,q_s)$.
Thus ${\bf SU}/\Gamma$ has finite volume, and moreover, 
if $s\neq r$, the quotient ${\bf SU}/\Gamma$ is compact, or equivalently
$\Gamma$ contains no unipotent elements (\cite{BH}).

If ${\bf K}$ denotes the maximal compact subgroup of ${\bf SU}$, then
the arithmetic groups described above determine finite volume
quotients of the symmetric space ${\bf SU}/{\bf K}$. In fact the full
group of holomorphic isometries is obtained by projectivizing these
groups; i.e. $\Gamma$ projects to an arithmetic lattice in
$\P{\bf SU} = \P\SU(p_1,q_1)\times \ldots \times \P\SU(p_s,q_s)$ (see \cite{Bo1} and \cite{BH}). 
Notice that for each $p_i,q_i$, there is a natural epimorphism
$\SU(p_i,q_i)\rightarrow \P\SU(p_i,q_i)$ whose
kernel consists of $N$th-roots of unity, and in particular is 
finite.

\subsection{} We maintain the notation of the previous subsection.
Let $\mathcal V$
denote the set of non-archimedean places of $k$. If $\cal P$
is a prime ideal in ${\mathcal O}_k$, we will write ${\cal \nu}_{\cal P}$
for the place in $\mathcal V$
associated to $\cal P$, and often simply write
$\nu$. The theory of the group
$\mathcal G$ over the local fields $k_\nu$ is well-understood and we summarize
what is needed for us (see \cite{PR} 
Chapter 2.3.3, \cite{Ti1} and \cite{Ti}).

Suppose that $L/F$ is a finite extension of number fields, with
rings of integers 
${\mathcal O}_L$ and ${\mathcal O}_F$
respectively. Let $\nu$ be
a place associated to a prime ideal 
${\cal P} \subset {\mathcal O}_F$.
Then
the behavior of $\cal P$ in $L/F$ is determined by how the
${\mathcal O}_L$-ideal ${\cal P}{\mathcal O}_L$
factorizes.  We say that $\nu$ (or the prime
$\cal P$) {\em splits completely} in $L/F$
if 
${\cal P}{\mathcal O}_L$
decomposes 
as a product of precisely $[L:F]$ 
pairwise distinct
prime ideals in 
${\mathcal O}_L$
(each of norm $q$
the rational prime lying below $\cal P$).

Consider the degree 2 extension $K/k$, and so a $k$-prime either remains
prime 
in
$K$, is ramified in $K$ or splits into a product of two 
distinct primes. The structure of ${\mathcal G}(k_\nu)$ depends on the
splitting type described above. Briefly, in the first two cases, the
persistence of the quadratic extension locally is enough to show that
${\mathcal G}(k_\nu)$ is a special unitary group. However,
if $\nu$ splits as a product of two primes
in the quadratic extension $K/k$, then 
$k_\nu \otimes_k K \cong k_\nu \times k_\nu$ is not a quadratic field extension
of $k_\nu$. Using this, it can be shown that
$${\mathcal G}(k_\nu)\cong\{(A,B)\in \SL(N,k_\nu)\times \SL(N,k_\nu)~:~
                        A={H^{-1}}^tB^{-1}H\}\cong \SL(N,k_\nu).$$

For more details see the discussion in 
Chapter 2.3.3
\cite{PR} , or \cite{Ti1} p. 55
and \cite{Ti}.  We summarize what is needed from this discussion  in
the 
       following

\begin{theorem}
\label{local}
Suppose that $q$ is a rational prime that splits completely to $K$, and $\nu$ 
a place of $k$ dividing $q$. Then
${\mathcal G}(k_\nu)\cong \SL(N,k_\nu) \cong \SL(N,{\bf Q}_q)$.\end{theorem}

The last isomorphism in Theorem \ref{local} follows from the fact
that for the places $\nu$ in Theorem \ref{local}, $k_\nu\cong {\bf Q}_q$.
That there are infinitely many such primes $q$ follows 
from Cebotarev's density theorem.

For all but finitely many primes 
${\cal P}\subset {\mathcal O}_k$, 
we can also consider $\mathcal G$ as an algebraic group
over the residue class field ${\bf F}_\nu ={\bf F}_{\cal P} \cong {\bf F}_q$ 
(see \cite{PR} pp. 142--143). Moreover, by \cite{PR} 
Chapter 3 Proposition 3.20,
for these primes the reduction
map ${\mathcal G}({\mathcal O}_k)\rightarrow {\mathcal G}({\bf F}_{\cal P})$ is a
surjective
homomorphism.  
Thus,
together with Theorem \ref{local} we deduce:

\begin{corollary}
\label{residue}
Suppose that $q$ is a rational
prime that splits completely to $K$, and $\cal P$ 
a $k$-prime dividing $q$.  Then for all but finitely many such
primes $\cal P$, ${\mathcal G}({\bf F}_{\cal P})\cong \SL(N,q)$.\end{corollary}

\subsection{} 
We continue with the notation above.  Being an algebraic subgroup of
$\SL(N,{\bf C})$, 
$\mathcal G$ comes equipped with the Zariski topology, and so 
in particular is Zariski closed
by definition. It also has the {\em analytic} (``usual'') topology arising
from the subspace topology inherited from $\SL(N,{\bf C})$. Thus
given a subgroup $D<{\mathcal G}$ we can talk about its Zariski
closure and analytic closure. Furthermore ${\mathcal G}({\bf R})$
is a Lie group and a real algebraic
group, and as such we can talk about the real Zariski
closure and analytic closure of subgroups $D<{\mathcal G}({\bf R})$.  We
collect some facts about the interplay between these topologies on 
these groups and their subgroups that will be used.

The following lemma is due to Chevalley (see \cite{Wit} Prop. 4.6.1).

\begin{lemma}
\label{chevally}
Let $D<\SU(N)$ be a subgroup. Then $D$ is Zariski closed in $\SU(N)$ if and 
only if it is analytically closed.\end{lemma}

\begin{lemma}
\label{biggerZD}
Suppose that $D<{\mathcal G}(k)$ is (real) Zariski dense in ${\mathcal G}({\bf R})$, 
then $D$ is Zariski dense in $\mathcal G$.\end{lemma}

\noindent{\bf Proof:}~Let $Z$ denote the Zariski closure of $D$ in $\mathcal G$.
Since $D<{\mathcal G}(k)$, $Z$ is an algebraic subgroup of $\mathcal G$ 
defined over $k$ (see \cite{Bo} Chapters I.1.3, and AG 14.4 for example).
Hence $Z({\bf R})<{\mathcal G}({\bf R})$ are real algebraic groups defined
over $k$, and so are real Zariski closed sets. But the Zariski closure of
$D$ in ${\mathcal G}({\bf R})$ is ${\mathcal G}({\bf R})$, and so it follows that 
$Z({\bf R})={\mathcal G}({\bf R})$.

Now viewed as real algebraic groups, the groups
${\mathcal G}({\bf R})$ and $Z({\bf R})$ are defined over $k$. The
algebraic groups $\mathcal G$ and $Z$ are also defined over $k$
and are simply the complexifications
of these real algebraic groups. Thus the ideals of polynomials defining
${\mathcal G}({\bf R})$ and $\mathcal G$ (resp. $Z({\bf R})$ and $Z$) agree. 
{}From this it follows that $Z=G$ as required.\qed

\subsection{} We will apply Strong Approximation, and
in particular, a corollary of
Theorem 10.5 of \cite{Wei} (see also \cite{No}).
Note that $\mathcal G$ 
is an absolutely almost simple simply connected algebraic group defined over
$k$ (i.e the only proper normal algebraic subgroups of $\mathcal G$ are finite)
which is required in \cite{Wei}.

For convenience we state the main
consequence of Theorem 10.5 and Corollary 10.6 of \cite{Wei} (see also  the discussion
in Window 9 of \cite{LS}) in our context.
\begin{definition} {\em The {\em adjoint trace field}
of a subgroup  $D<{\mathcal G}(k)$ is defined to be the field 
${\bf Q}(\{\tr(\Ad~\gamma):\gamma\in D\}$. Here $\Ad$ denotes the adjoint
representation of $\mathcal G$ on its Lie algebra.}
\end{definition}

\begin{theorem}
\label{weis}
Let $\mathcal G$ be as above, and let
$D<{\mathcal G}(k)$ be a finitely generated
Zariski dense subgroup of ${\mathcal G}$ such that
the adjoint trace field of $D$
is $k$.
Then for all but finitely many $k$-primes $\cal P$, the 
reduction homomorphism
$\pi_{\cal P} : D \rightarrow {\mathcal G}({\bf F}_{\cal P})$
is surjective.\end{theorem}

\noindent{\bf Proof:}~We briefly discuss how this is deduced from
Theorem 10.5 and Corollary 10.6 of \cite{Wei}.  Since $D$ is finitely
generated, apart from a finite set of places in 
$\mathcal V$, the image of $D$
(which we will identify with $D$) under the embedding ${\mathcal
  G}(k)\hookrightarrow {\mathcal G}(k_\nu)$, lies in the subgroup ${\mathcal
  G}({\mathcal O}_{k_\nu})$.  Now the conclusion of Corollary 10.6 of
\cite{Wei} states that there is a (perhaps different) finite set
$T \subset \mathcal V$
so that the closure of $D$ in the restricted direct
product group 
      $\prod_{{\mathcal V}\setminus T}{\mathcal G}({\mathcal O}_{k_\nu})$ 
is
open.  That this closure is open, in particular implies that 
for all ${\nu } \in {\mathcal V} \setminus T$, the closure of $D$ in the $\nu$-adic topology
is all of ${\mathcal G}({\mathcal O}_{k_\nu})$.  It follows that the
associated reduction homomorphism 
$\pi_{\cal P}$ is surjective.\qed\\[\baselineskip]
\noindent 
Theorem \ref{weis} together with Corollary \ref{residue} now 
shows the following:

\begin{corollary}
\label{weis2}
Let $D<{\mathcal G}(k)$ be a finitely generated
Zariski dense subgroup of $\mathcal G$ such that
the 
adjoint trace field of $D$
coincides with $k$.
Then there are infinitely many $k$-primes $\cal P$ of norm $q$ a prime in $\bf Z$, 
for which 
the reduction homomorphism $\pi_{\cal P} : D \rightarrow \SL(N,q)$
is surjective.\end{corollary}

\begin{remark}\label{2.8}{\em It is clear from the proof that Theorem~\ref{weis}
    and Corollary~\ref{weis2} also hold if 
we assume the adjoint trace field is a subfield $\ell \subset k$,
provided that $\mathcal G$ can be defined over
    $\ell$, and $D$ lies in the $\ell$-points of $\mathcal G$ (the point
being that a rational prime that splits completely in $k$ must split completely in $\ell$). 
This observation will allow for a shortcut in our proof of
Theorem~\ref{main} in \S\ref{sec4}.
}\end{remark}

\section{The $\SO(3)$-TQFT representations} \label{sec3}

We briefly recall some of the background from 
the $SO(3)$-TQFT
constructed in
\cite{BHMV} and its integral version constructed in \cite{GM}.  We
also record some consequences of this and \cite{LW} that we will make use of.
{}From now on, 
we only consider the case where the prime $p$ satisfies
$p\equiv 3 \pmod 4 $.

\begin{remark}{\em It is possible to make everything what
follows work for all odd primes, but doing so requires
some modifications and some extra arguments in the case $p\equiv 1
\pmod 4 $. Since primes 
$p \equiv 3 \pmod 4 $ are enough to prove Theorems~\ref{main} and
\ref{main2}, we prefer to restrict to that case for simplicity.
}\end{remark}

Let $\Sigma$ be  a
compact oriented surface of genus $g$ without boundary, and let
$\Gamma_g$ be its 
mapping class group. The integral $SO(3)$-TQFT constructed in
\cite{GM} provides a representation of a central extension $\widetilde
\Gamma_g$  of $\Gamma_g$ 
by $\BZ$
on a free lattice ({\em i.e.} a free module of 
finite rank)
${\mathcal S}_p(\Sigma)$ over the
ring of cyclotomic integers $\BZ[\zeta_p]$ : $$\rho_p \,:\, \widetilde
\Gamma_g \longrightarrow \GL({\mathcal S}_p(\Sigma))\simeq 
\GL(N_g(p),\BZ[\zeta_p])~,$$ 
where $N_g(p)$ is the rank of ${\mathcal S}_p(\Sigma)$.
We
refer to this representation as the SO(3)-TQFT-representation. 
Some results and conjectures about this representation are discussed
in \cite{Ma}.
We also denote by
$V_p(\Si)$ the 
$K$-vector space ${\mathcal S}_p(\Sigma)\otimes K$ where
$K=\BQ(\zeta_p)$ as in \S \ref{sec2}.
The $V_p$-theory is a
version of the Reshetikhin-Turaev TQFT
associated with the Lie group $SO(3)$, and we think of ${\mathcal S}_p$ as an
integral refinement of that theory (see 
\cite{GM} for more
details). The rank $N_g(p)$ of ${\mathcal S}_p(
\Sigma)$ is given by a Verlinde-type formula and goes to infinity as
$p\rightarrow \infty$. The construction uses the skein theory of the
Kauffman bracket with Kauffman's skein variable $A$ specialized to
$A=-\zeta_p^{(p+1)/2}$. Note that $A^2=\zeta_p$ and $A$ is a primitive $2p$-th root of
unity.

We assume $g\geq 3$, so that $\Gamma_g$ is perfect and $H^2(\Gamma_g;
\BZ)\simeq \BZ$. It is customary in TQFT to take the extension $\widetilde
\Gamma_g$ to be isomorphic to Meyer's signature extension, whose cohomology class is  $4$ times a generator of $H^2(\Gamma_g;
\BZ)\simeq \BZ$. However, in this paper we take $\widetilde
\Gamma_g$ so that the cohomology class $[\widetilde
\Gamma_g]$ is a generator of $H^2(\Gamma_g;
\BZ)$. Thus our $\widetilde
\Gamma_g$ is (isomorphic to) an index four subgroup of the signature
extension. The advantage of this choice is that $\widetilde
\Gamma_g$ is a perfect group. In fact, $\widetilde
\Gamma_g$ is a universal central extension of $\Gamma_g$ for $g\geq
4$.

\begin{remark} {\em The are various constructions of these central
    extensions of the mapping class group from the  TQFT point
 of view. We will not discuss them here as the details are not relevant
 for this paper. To be specific, we follow the approach in  \cite{GM2}, except for notation:
 our  $\widetilde
\Gamma_g$ is denoted by  $\widetilde
\Gamma_g^{++}$ in  \cite{GM2}.
Up to isomorphism, this is the same as the extension denoted by
$\widetilde\Gamma_1$ in \cite{MR}.
}\end{remark}

The
generator of the kernel of $\widetilde \Gamma_g\rightarrow \Gamma_{g}
$ acts as multiplication by  
            $\zeta_p^{-6}$
on ${\mathcal
  S}_p(\Sigma)$.
(This is the fourth power of the number $\kappa$ as given in
\cite[\S 11]{GM2}.)
Since 
$\zeta_p^{-6}\neq 1$,
the TQFT-representation
$\rho_p$ induces only a projective representation of the mapping class
group $\Gamma_{g} $.

\begin{notation} 
Henceforth, the
image group $\rho_p(\widetilde{\Gamma}_g)$ will be denoted by
$\Delta_g$.
\end{notation}

\begin{remark}{\em The following observation will be used in the
    proof of our main theorem: 
If we have a surjection from $\Delta_g$ to a finite
group $H$, the induced surjection 
$$\widetilde{\Gamma}_g\twoheadrightarrow
\Delta_g\twoheadrightarrow H$$ will factor through a surjection
$\Gamma_g \twoheadrightarrow H$ as soon as $H$ has no non-trivial central
element of order $p$ (because $\widetilde{\Gamma}_g\rightarrow
\Gamma_g$ is a central extension and the generator of its kernel is
sent to an element of order $p$ in $\Delta_g$). In particular if $H$ 
has trivial center this will hold.
}\end{remark}

We now refine the strategy outlined in \S \ref{sec1}.
First, as   
observed in \cite{DW}, the map $$\det
\circ \rho_p : \widetilde \Gamma_g \longrightarrow
\BZ[\zeta_p]^{\times}$$ is trivial,  
since $\widetilde\Gamma_g$ is perfect. Therefore the group
$\Delta_g=\rho_p(\widetilde{\Gamma}_g)$ is contained in a
special linear group: 
$$\Delta_g\  \subset \ \SL({\mathcal S}_p(\Sigma))\simeq 
      \SL(N,\BZ[\zeta_p])~,$$
where $N=N_g(p)$.
The primes $\tilde q$ in $\BZ[\zeta_p]$ 
mentioned
in \S \ref{sec1} lie above
those rational primes $q$ which split completely in
$\BZ[\zeta_p]$. For every such prime $\tilde q$ of $\BZ[\zeta_p]$ lying
over $q$, we can consider the group 
$$\pi_{\tilde q}(\Delta_g) \subset \SL(N,q)~,$$ where $\pi_{\tilde q}$ is the reduction homomorphism from
$\SL(N,\BZ[\zeta_p])$ to $\SL(N,q)$ induced by the isomorphism
$\BZ[\zeta_p]/\tilde q\simeq \BF_q$.
The
 key step in the proof of Theorem~\ref{main2} 
is to establish
 \begin{equation} \label{eqpi} \pi_{\tilde q}(\Delta_g) = \SL(N,q)
\end{equation} for all but finitely many
 such $\tilde q$. This will be an application of
 Corollary~\ref{weis2} and is described in \S \ref{sec4}. 
Thus, as announced in \S \ref{sec1}, we will have surjections $\Delta_g\twoheadrightarrow \SL(N,q)$ for
infinitely many primes q.
The surjections
 $\Gamma_g \twoheadrightarrow \PSL(N,q)$ follow easily and this 
will complete the proof of Theorem~\ref{main2}.

\begin{remark}{\em As far as proving the equality (\ref{eqpi}) for all
    but finitely many $\tilde q$, we do not actually need Integral
    TQFT. Here 
by Integral TQFT we mean
the fact that the TQFT-representation $\rho_p$
    preserves the lattice ${\mathcal S}_p(\Sigma)$ inside the
    TQFT-vector space  $V_p(\Si)={\mathcal S}_p(\Sigma)\otimes K$,
    which we used to arrange  that $\Delta_g
    =\rho_p(\widetilde\Gamma_g)
                                $ lies in $\SL(N,\BZ[\zeta_p])$ rather
    than just in $\SL(N,\BQ(\zeta_p))$. The point is that even if  $\Delta_g
   $ is only known to lie in $\SL(N,\BQ(\zeta_p))$,  we can still define $\pi_{\tilde
     q}(\Delta_g)$ for all but finitely many $\tilde q$ (because
   $\Delta_g$ is a finitely generated group, and so involves only finitely
   many primes in the denominators of its matrix entries). This
   is enough for our application of  Corollary~\ref{weis2}. On
   the other hand, it is interesting to know that the group $\pi_{\tilde
     q}(\Delta_g)$ is always defined, and one may ask which are the
   exceptional primes $q$ (if any) for which this
   group is strictly smaller than $\SL(N,q)$?
}\end{remark}  

In order to apply Corollary~\ref{weis2} to $\Delta_g$, we need to describe the Zariski closure
 of $\Delta_g$ as an algebraic group defined over a number field, which we will now do in the
 remainder of this section. The first step is to observe that  
$\Delta_g$ lies in a (special) unitary group.
This is because, as always in TQFT, the representation $\rho_p$
preserves a non-degenerate Hermitian form. Here, conjugation is given by
$\overline{\zeta_p} = \zeta_p^{-1}$. 
Let us denote by $H_p$ the 
Hermitian form 
on the vector space $V_p(\Si)$
defined in
\cite{BHMV}. There is a basis of  $V_p(\Si)$
which is orthogonal for this form; moreover the diagonal
terms of the matrix of $H_p$ in this basis lie in the
maximal real subfield $k$. 
Explicit formulas for these diagonal terms are
given in 
\cite[Theorem~4.11]{BHMV}. 

\begin{remark}{\em (i) Note that $H_p$ is denoted by $\langle\ ,\
\rangle_\Si$ in 
\cite{BHMV}. We are using here that $p\equiv 3\pmod 4$, because in
this case the coefficient $\eta=\langle S^3 \rangle_p$ which appears in
\cite[Theorem~4.11]{BHMV} lies in $k$. Indeed, we have $\overline \eta
=\eta$ and it is shown in \cite[Lemma
4.1(ii)]{GMW} that $\eta^{-1}$ (which is called ${\mathcal D}$ in
\cite{GMW,GM}) lies in $\BZ[\zeta_p]$.

(ii) It is shown in \cite{GMW,GM} that one can rescale the hermitian form so that its values on the lattice ${\mathcal S}_p(
\Sigma)$ lie in $\BZ[\zeta_p]$ (it suffices to multiply the form by
the number 
$\mathcal D$).
}\end{remark}

 Thus the Hermitian
form $H_p$ is defined over $k$. As in \S 2.1, let $\mathcal G$ be the
group $\SU(V_p(\Si), H_p)$; this is an algebraic group $\mathcal G$
defined over $k$, and $$\Delta_g=\rho_p(\widetilde \Gamma_g) \subset
{\mathcal G} (k).$$

The signature of the Hermitian form $H_p$
depends on the choice of
$\zeta_p$ in $\BC$. For 
the choice $$ A=i^p e^{2\pi i/4p}, \ \ \zeta_p = A^2 = (-1)^p e^{2\pi i/2p}
= (e^{2\pi i/p})^{(p+1)/2}$$
the form $H_p$ is positive
definite so that ${\mathcal G} ({\bf R})$ is isomorphic to the usual
special unitary group $\SU(N)$  where $N=N_g(p)=\rk\ {\mathcal S}_p(
\Sigma) = \dim \ V_p(\Si)$.  
For other choices of  $\zeta_p$ in $\BC$ the form is
typically indefinite as soon as the genus is at least two \cite[Remark
4.12]{BHMV}.

We now recall the following result of Larsen and Wang \cite{LW}. 

\begin{theorem}\cite{LW}
\label{LarsenWang}
For the choice of root of unity given above, $\Delta_g$ projects
to a subset of $\PSU(N)$ that is dense in the analytic
topology.\end{theorem}

\begin{remark} {\em Larsen and Wang actually take $A= i
    e^{2\pi i/4p}$ if $p\equiv 3 \pmod 4$. This differs from our
    choice of $A$ by a sign. The explanation is that Larsen and Wang
take
$A$ to be a primitive $p$-th root whereas in the
    skein-theoretic approach to TQFT
of \cite{BHMV} which we are using, $A$ must be a primitive
    $2p$-th root 
(essentially because in the axiomatics of \cite{BHMV}, Kauffman's
skein variable must be 
$A$     rather than $-A$).
However, 
in the $\SO(3)$-case, 
 the TQFT-representation $\rho_p$ of $\widetilde\Gamma_g$
only depends on $A^2=\zeta_p$, 
    so the sign of $A$ is, in fact, irrelevant here.
}\end{remark}

Since $\SU(N) \rightarrow \PSU(N)$ is a finite covering, 
a corollary of Theorem \ref{LarsenWang} is:

\begin{corollary}
\label{SUNdense}
With the notation as above, $\Delta_g$ is a dense
subgroup of $\SU(N)$ in the analytic topology.\end{corollary}
We also see from this discussion, and
that contained in \S 2.1, that $\Delta_g$
contains no unipotent elements.
\begin{corollary}
\label{allzariski}
In the notation above, $\Delta_g$ is Zariski dense in the algebraic
group $\mathcal G$.\end{corollary}

\noindent{\bf Proof:} This follows applying Lemma \ref{chevally}, and
Lemma \ref{biggerZD}.

\begin{remark} {\em (1) 
At present 
it remains open
whether the index of  $\Delta_g$ in 
the arithmetic group 
$\Gamma \simeq \mathcal G ({\mathcal O}_k)$
(see the discussion in \S 2.1)
is finite
or infinite.
If this index were finite
then $\Delta_g$ would have been arithmetic and so Zariski density would
follow from Borel density.

\noindent (2) We also note that Zariski density at other embeddings of $k$ into
    $\bf R$ follows easily from this, but we will not need to make use of this fact.
}\end{remark}

\section{Proof of the main results} \label{sec4}

\subsection{Proof of Theorems~\ref{main} and \ref{main2}} 

Fixing $g\geq 3$ and a prime $p\equiv 3 \pmod 4$, the discussion in 
\S \ref{sec3} 
shows that 
we have a representation $\rho_p$ of 
$\widetilde{\Gamma}_g$
whose image $\Delta_g$ lies in the 
$k$-points
of the algebraic
group $\mathcal G$ defined over 
$k$,  where $k$ is the maximal real subfield of the cyclotomic field
$K=\BQ(\zeta_p)$, with the root of unity $\zeta_p\in \BC$ chosen so that ${\mathcal G}({\bf R})\cong
\SU(N)$.
Moreover, $\Delta_g$ is Zariski dense in ${\mathcal G}$.
We wish to apply Corollary~\ref{weis2} to this situation. Notice that
all the hypotheses of this corollary are already satisfied, except the
hypothesis about the adjoint trace field. Denote the adjoint trace
field of $\Delta_g$ by 
$$\ell={\bf Q}(\tr(\Ad\,\gamma): \gamma \in \Delta_g)~.$$
As observed in Remark~\ref{2.8}, it is enough to check that $\mathcal G$ can be
defined over $\ell$, 
and that $\Delta_g$ lies in the $\ell$-points of  $\mathcal G$. 
This is the content of Proposition~\ref{41} below, which we prove next.

\begin{lemma} 
\label{ellisasubset}
We have $\ell
\subset k$.
\end{lemma} 

\begin{proof} As in \S 2.1 and 2.2, we are considering $\mathcal G$ as a 
$k$-algebraic subgroup
of $\SL(2N)$. 
We denote the adjoint group $\Ad \,\mathcal G $ by $\Gad$.
 Since $\Delta_g\subset {\mathcal G}(k)$, we have
$\Ad\,\gamma \in \Gad(k)$ for all $\gamma \in \Delta_g$. This shows $\ell
\subset k$.
\end{proof}

\begin{proposition} \label{41} The group $\mathcal G$ 
can be defined over $\ell$, and one has $\Delta_g\subset {\mathcal G}(\ell)$.
\end{proposition}

\noindent{\bf Proof:}
By Vinberg's theorem \cite[Theorem 1]{Vi} (see also
\cite[(2.5.1)]{Mo}), Zariski 
density of $\Delta_g$ in ${\mathcal G}$ together with ${\bf
  Q}(\tr(\Ad\,\gamma): \gamma \in \Delta_g) =\ell$ imply that there
is an $\ell$-structure on $\Gad$ ({\em i.e.}, the group $\Gad$ can be
defined over $\ell$) so that $\Ad\,\Delta_g \subset
\Gad(\ell)$. Since $\mathcal G$ is simply connected, by a well-known result of Borel-Tits \cite{BT} (see also \cite[Section~2.2]{PR}), this $\ell$-structure on $\Gad$ can be lifted to an
$\ell$-structure on $\mathcal G$ so that the canonical projection
$\pi: \mathcal G \rightarrow \Gad$ is defined over $\ell$.

As already mentioned, Vinberg's theorem also gives that $\Ad\,\Delta_g \subset
\Gad(\ell)$. We must show that, in fact, $\Delta_g\subset {\mathcal
  G}(\ell)$. This is, however,  not a formal consequence of Vinberg's theorem,
but uses the fact that $\Delta_g$ is perfect. 
We proceed as follows. To show that $\Delta_g\subset {\mathcal
  G}(\ell)$, we will show that $\sigma(\gamma)=\gamma$ for every
$\gamma\in \Delta_g \subset {\mathcal G}(k) $ and $\sigma
\in \Gal(k / \ell)$ 
(recall that Lemma \ref{ellisasubset} shows that
$\ell \subset k$).  Consider the exact sequence 
$$C(k)
\rightarrow {\mathcal G} (k) \,\mapright{\pi} \,\Gad (k)$$ where $C$ is
the center of $\mathcal G$. Since $\Ad\,\gamma\in \Gad(\ell)$ for
$\gamma\in \Delta_g$, we have $$\pi(\sigma(\gamma))= \sigma
(\pi(\gamma)) = \pi(\gamma)$$ for every $\sigma\in\Gal(k/\ell)$. Hence the function $f_\gamma$ defined by
$$f_\gamma(\sigma) = 
\gamma \, \sigma(\gamma^{-1})$$ is a $C(k)$-valued $1$-cocycle on $\Gal(k
/ \ell)$. (One easily checks the cocycle condition
$f_\gamma(\sigma_1\sigma_2)=f_\gamma(\sigma_1)
\sigma_1(f_\gamma(\sigma_2))$.) Let $Z^1(\Gal(k
/ \ell); C(k))$ denote the space of such cocycles. It is an abelian
group (since $C(k)$ is abelian.) Moreover, the assignment $\gamma
\mapsto f_\gamma$ is a group
homomorphism from $\Delta_g$ to $Z^1(\Gal(k
/ \ell); C(k))$. But since $\Delta_g$ is perfect, this homomorphism is
trivial; in other words, we have $f_\gamma=1$ for all $\gamma\in
\Delta_g$. This shows that $\Delta_g\subset {\mathcal
  G}(\ell)$, as asserted. \qed
\begin{remark}{\em A natural question at this point is whether
    $\ell=k$. As far as the proofs of Theorems~\ref{main} and
    \ref{main2} are concerned, whether the answer is in the
    affirmative or not, does not matter
because, as observed in Remark~\ref{2.8}, we can simply apply
Corollary~\ref{weis2} with $\ell$ in 
place of $k$.
However, for completeness,
    and because it seems worthwhile recording, we will prove that indeed
    $\ell=k$ (using Proposition \ref{41}) in \S 4.3.}\end{remark}

We can now 
give the proof of Theorem \ref{main2} which is restated below for convenience.

\begin{theorem}
\label{pslquots}
For each $g\geq 3$, there exists infinitely many $N$ such that
for each such $N$, there exists infinitely
many primes $q$ such that $\Gamma_g$ surjects $\PSL(N,q)$.\end{theorem}

\noindent{\bf Proof:}~Fixing $g\geq 3$, the discussion in 
\S \ref{sec3} together with Proposition~\ref{41} 
shows that for every prime $p\equiv 3 \pmod 4$ 
we have a representation $\rho_p$ of 
$\widetilde{\Gamma}_g$
whose image $\Delta_g$ lies in the $\ell$-points of the algebraic
group $\mathcal G$ defined over $\ell$, where $\ell$ is a finite
Galois extension of $\BQ$ and ${\mathcal G}({\bf R})\cong \SU(N)$,
with $N=N_g(p)$ going to infinity as $p\rightarrow \infty$. 
Moreover, $\Delta_g$ is Zariski dense in ${\mathcal G}$ and 
its adjoint trace field is $\ell$.

Fixing such a dimension $N$ as above, we deduce from Corollary \ref{weis2} 
that there are infinitely many
rational primes $q$ such
that $\Delta_g$ surjects the groups $\SL(N,q)$.  Now
quotienting out by the center of $\SL(N,q)$ gives surjections
of $\Delta_g$ onto $\PSL(N,q)$. 
As remarked in Remark 3.3,
the induced homomorphisms 
$\widetilde{\Gamma}_g \rightarrow \PSL(N,q)$
will factor through $\Gamma_g$ since $\PSL(N,q)$ 
has trivial center.\qed\\[\baselineskip]
The proof of Theorem \ref{main} will be completed by 
the following basic fact about embedding
finite groups in the groups $\PSL(N,q)$.

\begin{lemma}
\label{symassubgroups}
Let $H$ be a finite group, then there exists an integer $N$ such that
for all odd primes $q$, $H$ is isomorphic
to a subgroup of $\PSL(N,q)$.
\end{lemma}

\noindent{\bf Proof:}~By Cayley's theorem, every finite group embeds
in a symmetric group. Thus it suffices to prove the lemma for symmetric
groups $S_n$.  Note first that $\PSL(N,q)$ has even order
and so will trivially contain a copy of $S_2$ (being isomorphic to
the cyclic group of order $2$).
Thus we can assume that $n\geq 3$. We first prove that $S_n$ 
injects
into $\SL(N,q)$ (for large enough $N$). 

To that end, recall that the standard permutation representation of
$S_n$ injects $S_n\hookrightarrow \GL(n,{\bf Z})$. Furthermore,
$\GL(n,{\bf Z})$ can be embedded in $\SL(n+1,{\bf Z})$ by sending
$g\in \GL(n,{\bf Z})$ to the element
$$
\begin{pmatrix} g & 0\\ 0& \epsilon (g)\end{pmatrix}~,
$$
\noindent where $\epsilon (g) =\pm 1$ depending on whether $\det (g)= \pm 1$. 

It is a well-known result of Minkowski that the kernels of the homomorphisms
$\SL(N,{\bf Z})\rightarrow \SL(N,q)$ are torsion-free
for $q$ an odd prime (see \cite{PR} Lemma 4.19). 
Hence the copies of $S_n$ constructed above
inject in $\SL(N,q)$ as required.

To pass to $\PSL(N,q)$, we simply note that
$\PSL(N,q)$ is the central quotient 
of $\SL(N,q)$, and
the center of $S_n$ is trivial 
for $n\geq 3$. Hence $S_n$ will inject into $\PSL(N,q)$.\qed\\[\baselineskip]

\subsection{The case of the Torelli group}

We now discuss the case of the Torelli subgroup ({\em i.e.,} 
the kernel of the
homomorphism $\Gamma_g \rightarrow \Sp(2g,{\bf Z})$). We will denote the
Torelli group by ${\cal I}_g$. It is shown in \cite{Jo} that ${\cal I}_g$
is finitely generated for $g \geq 3$ and in \cite{Me} that ${\cal I}_2$
is an infinitely generated free group.

\begin{theorem}
\label{torelli}
For each $g\geq 2$, there exists infinitely many $N$ such that
for each such $N$, there exists infinitely
many primes $q$ such that ${\cal I}_g$ surjects $\PSL(N,q)$.\end{theorem}

\noindent{\bf Proof:}~As noted above ${\cal I}_2$ is an infinitely
generated free group and so the result easily holds in this case.
Thus we fix a $g\geq 3$, and 
consider a surjection  
              $ f: \Gamma_g \twoheadrightarrow \PSL(N,q)$
as constructed in Theorem \ref{pslquots}.
Since  ${\cal I}_g$ is normal in $\Gamma_g$, 
the image $f({\cal I}_g)$  
in $\PSL(N,q)$ will also be normal.  The groups
$\PSL(N,q)$ are simple, and so 
$f({\cal I}_g)$ 
is either trivial or $\PSL(N,q)$.

We claim that for $N$ large enough the image must be $\PSL(N,q)$.  For
suppose not, then for 
some arbitrarily large $N$
the image 
$f({\cal I}_g)$
will be trivial, and so the epimorphisms 
$f:\Gamma_g \twoheadrightarrow \PSL(N,q)$ 
will factor through $\Sp(2g,{\bf Z})$. However, as mentioned in \S \ref{sec1},
$\Sp(2g,{\bf Z})$ has the Congruence Subgroup Property and so cannot surject
the groups $\PSL(N,q)$ (for $N$ large).\qed

\subsection{The adjoint trace field}\label{adf}

We 
briefly discuss how to deduce that the adjoint trace field $\ell=
{\bf Q}(\tr(\Ad\,\gamma): \gamma \in \Delta_g)$ is equal to
$k$. (Recall that $k$ is the maximal real subfield of the cyclotomic
field $K=\BQ(\zeta_p)$.) We proceed as follows. 

{}From Lemma~\ref{ellisasubset} and
Proposition~\ref{41}, we have that $\ell$ is a subfield of $k$ so
that $\mathcal G$ can be defined over $\ell$, and $\Delta_g \subset
{\mathcal G}(\ell)$. The group $\mathcal G$, when considered as
defined over $\ell$, is an $\ell$-form of $\SU(N)$. By the
classification of forms of $\SU(N)$ over number fields 
          \cite[\S (2.3.3) and (2.3.4)]{PR},
there is a central simple algebra $A$, with
center $L$ a quadratic field extension of $\ell$, so that
$${\mathcal G}(\ell) = \{ x\in A \,\, \vert\,\, x \,\tau(x) =1, \ \Nrd(x)=1\}~,$$ where $\tau$ is an (anti-)involution of $A$ of the second kind, and $
\Nrd$ is the reduced norm. Therefore for all $\gamma\in\Delta_g\subset
{\mathcal G}(\ell)$, we have $$\Trd(\gamma)\in L~,$$ where $\Trd$ is
the reduced trace.  When we extend scalars from $\ell$ to $k$, our
group $\mathcal G$ viewed as an $\ell$-group becomes $k$-isomorphic to
our original $k$-group $\mathcal G$. Thus $$A\otimes K \simeq M_N(K)$$
(where $N=N_g(p)$ is the dimension of the $K$-vector space
$V_p(\Sigma)$), and the reduced trace $\Trd(\gamma)$ is (strictly by
definition) nothing but the ordinary trace of $\gamma$ viewed as an
element of $M_N(K)$. For $\gamma\in\Delta_g=\rho_p(\widetilde
\Gamma_g)$, this is the same as the trace of $\gamma$ acting on
$V_p(\Sigma)$.

Now recall that the generator of the kernel of the central extension $\widetilde \Gamma_g \rightarrow \Gamma_g$ acts as multiplication by a primitive $p$th root of unity on the vector space $V_p(\Sigma)$. Thus $\Delta_g=\rho_p(\widetilde \Gamma_g)$ contains an element $\gamma$ whose trace on $V_p(\Sigma)$ is $N$ times $\zeta_p$. Since this is the same as $\Trd(\gamma)$, and we know that $\Trd(\gamma)\in L$, it follows that $\zeta_p\in L$, hence $L=K$.  Since $\ell \subset k $ and $[L:\ell]=2$, this shows $\ell=k$.

\section{Comments} \label{sec5}

\noindent{\bf 1.}~As shown in \cite{LR} for example, if $\Gamma$ is
a finitely generated group that contains a non-abelian free group, and
$\Gamma$ is LERF ({\em i.e.,}
all finitely generated subgroups of $\Gamma$ are
closed in the profinite topology on $\Gamma$), then all finite groups
are involved in $\Gamma$. In the context of lattices in semi-simple Lie
groups, it is only in rank 1 that examples of LERF lattices are known,
although large classes of lattices in these rank 1 Lie groups are
known to have a slightly weaker separability property (see for
example \cite{ALR}, \cite{BHW}, and \cite{LR}). In higher rank
the expectation is that lattices will not be LERF, since the
expectation is that the Congruence Subgroup Property should hold for
these higher rank lattices.  As mentioned in \S \ref{sec1}, if the group
$\Gamma$ is an arithmetic lattice that has the Congruence Subgroup
Property, then the 
finite groups that are involved in $\Gamma$ 
are
restricted.

It is an easy fact that $\Gamma_g$ is not LERF (see Appendix A 
of \cite{LM}).\\[\baselineskip]
\noindent{\bf 2.}~Let $F_n$ denote a free group of rank $n$ and $\Out(F_n)$
denote its outer automorphism group.
The family of
groups $\Out(F_n)$, $n\geq 2$
are often studied in comparison to Mapping Class groups.  Typically, a theorem
about Mapping Class groups is reworked in the context of $\Out(F_n)$.
In regards to Theorem \ref{main}, it was already known
from \cite{Gil} that all finite groups are involved in $\Out(F_n)$.
Indeed, for $n\geq 3$, it is shown in \cite{Gil} that $\Out(F_n)$
is residually symmetric 
({\em i.e.,}
given $1\neq\alpha\in \Out(F_n)$ there is a finite
symmetric group $S_m$ and an epimorphism $\theta:\Out(F_n)\rightarrow S_m$
with $\theta(\alpha)\neq 1$).

Another proof that all finite groups are involved in $\Out(F_n)$ can be
deduced from \cite{GL} using methods similar to those used here.

%
%
%
%

\end{document}